\documentclass[12pt]{amsart}

\usepackage{amsfonts}
\usepackage{amssymb, latexsym, amsthm, amsmath, verbatim}
\usepackage{indentfirst}

\usepackage{hyperref}

\numberwithin{equation}{section}
\theoremstyle{definition}
\newtheorem{Thm}[equation]{Theorem}
\newtheorem{Prop}[equation]{Proposition}
\newtheorem{Cor}[equation]{Corollary}
\newtheorem{Lem}[equation]{Lemma}

\newtheorem{conj}[equation]{Conjecture}

\makeatletter
\def\imod#1{\allowbreak\mkern5mu{\operator@font mod}\,\,#1}
\makeatother

\begin{document}

\title[Eigenform Product Identities]{Eigenform Product Identities for Hilbert Modular Forms}
\author{Kirti Joshi and Yichao Zhang}
\address{Department of Mathematics, the University of Arizona, Tucson, AZ 85721-0089}
\email{kirti@math.arizona.edu}
\address{Department of Mathematics, the University of Arizona, Tucson, AZ 85721-0089}
\email{yichaozhang@math.arizona.edu,  zhangyichao2002@gmail.com}
\date{}

\begin{abstract} We prove that amongst all real quadratic fields and all spaces of Hilbert modular forms of full level and of weight $2$ or greater, the product of two Hecke eigenforms is not a Hecke eigenform except for finitely many real quadratic fields and finitely many weights. We show that for $\mathbb Q(\sqrt 5)$ there are exactly two such identities. 
\end{abstract}

\maketitle
\tableofcontents
\newcommand{\Z}{{\mathbb Z}} 
\newcommand{\Q}{{\mathbb Q}} 

\section{Introduction}
Let $E_k$ be the normalized Eisenstein series of weight $k$ on $SL_2(\Z)$. There are many classical identities between these Eisenstein series $E_k$ for different weights $k$, for instance
\begin{eqnarray}
E_8&=&120E_4^2\\
E_{10}&=&\frac{5040}{11}E_6E_4\\
\Delta_{16}&=&240E_4\Delta,
\end{eqnarray}
where $\Delta_{16}$  (resp. $\Delta$) is the unique, normalized cuspidal Hecke eigenform of weight $16$ (resp. $12$) on $SL_2(\Z)$ (the numerical constants in the above identities are  normalization constants).

These identities provide solutions to the equation
\begin{equation}\label{eq-iden}
g=f\cdot h
\end{equation}
in Hecke eigenforms.   For elliptic modular forms of full level, Duke \cite{duke1999product} and Ghate \cite{ghate2000monomial} independently considered this question and proved that there are precisely $16$ such identities (all of these identities were classically known). Let us note that by considering $q$-expansions it is immediate that a product of two or more normalized cuspidal Hecke eigenforms cannot be a Hecke eigenform. So in \eqref{eq-iden} at most one of $f,h$ can be cuspidal.
We say such an \textit{eigenform product identity holds trivially}, if the  dimension of the corresponding modular form space or the cusp form space for $g$ is equal to one. All of the $16$ identities hold trivially. The proofs of \cite{duke1999product,ghate2000monomial} use Rankin-Selberg convolution. Later Ghate \cite{ghate2002products} considered another type of eigenform product identities, where the eigenforms are a.e. Hecke eigenforms of weight $3$ or greater and of squarefree level, and proved that all such identities hold trivially. Emmons \cite{emmons2005products} considered $\Gamma_0(p)$, with $p\geq 5$ a prime, and classified eigenform product identities for eigenforms away from the level (eigenform for $T_m$ with $m$ coprime to $p$). 
Recently Johnson \cite{johnson2013hecke} considered such identities for $\Gamma_1(N)$ of weight $2$ or greater and found a complete list of $61$ eigenform identities, some of which hold non-trivially. In his thesis  Beyerl \cite{beyerl2012factoring}, for the full modular group, considered the question when the quotient of two Hecke eigenforms is a modular form.

Inspired by Johnson's  approach \cite{johnson2013hecke}, we consider this question  for Hilbert modular forms. We show that product of two Hecke eigenforms over a fixed real quadratic field can be another Hecke eigenform. For instance we show that for $F=\Q(\sqrt{5})$
\begin{eqnarray}
E_4&=& 60E_2^2,\\
h_8&=& 120 E_2\cdot h_6,
\end{eqnarray}
where $E_2=E_2(1,1),E_8=E_4(1,1)$ are Eisenstein series of parallel weight two (resp. four) with trivial characters, $h_6$ (resp. $h_8$) is the unique normalized cuspidal Hecke eigenform of parallel weight six (resp. eight) for $GL_2^+(\mathcal O_F)$ (see Theorem \ref{D=5}).

Hence identities of the type \eqref{eq-iden} exists for Hilbert modular forms. So it is natural to ask if there are only finitely many such identities amongst Hilbert modular forms. In this paper we will only consider Hilbert modular forms for Hilbert modular groups of full levels and answer this affirmatively.

In fact, a much stronger assertion is true for Hilbert modular forms. Explicitly, we prove  that \textit{amongst all real quadratic fields} $F$ the equation \eqref{eq-iden} has only finitely many solutions in Hecke eigenforms of full level and weights $2$ or greater:

\begin{Thm}\label{thm} Over all real quadratic number fields $F$ and all Hecke eigenforms for $\text{GL}_2^+(\mathcal O_F)$ of integral parallel weight $2$ or greater,
the equation 
$g=f\cdot h$ in the triple $(g,f,h)$ has only finitely many solutions.
\end{Thm}

Identities such as \eqref{eq-iden} provide relations between Fourier coefficients. One of the important observations of \cite{johnson2013hecke} is that relations between Fourier coefficients at small primes (or at powers of small primes) can be used to provide effective bounds for such identities. We adapt methods of \cite{johnson2013hecke} to the Hilbert modular form situation, but there are new features: for instance we exploit the discriminant of the  real quadratic field $F$, which manifests itself through its presence in the functional equation of $L$-functions, to effectively bound the number of the  real quadratic field for which product identities can exist. 
As in Johnson's treatment of the classical case, all bounds in this paper are effective and can be used to obtain a complete list of eigenform product identities, provided that we have the structure of the spaces of Hilbert modular forms of small weights for small $D$ (discriminant of $F$).  We are content with the concrete case for $\mathbb Q(\sqrt 5)$ for the moment, and prove that there are exactly two such identities (Theorem \ref{D=5}), using such effective bounds in the proof of the Theorem~\ref{thm}. As in the case of elliptic modular forms of full level, these two identities for $\mathbb Q(\sqrt 5)$ hold trivially.

We remark that for general levels and general narrow ray class characters, the conductors of the characters will appear in the L-values in question. To treat such general situation, one should consider more Fourier coefficients and obtain more equations in the weights to get around of the conductors and finally obstruct such identities. 

In Section 2 and 3, we set up the notations and provide the necessary background on Hilbert modular forms of full levels. In Section 4, we prove some formulas on Fourier coefficients of the product of two Hilbert modular forms and break up  Theorem~\ref{thm} into Theorem~\ref{thm1} and Theorem~\ref{thm2}, which will be proved in Section 5 and 6 respectively. In Section 7, we obtain the complete list of such identities for $\mathbb Q(\sqrt 5)$.
Finally, in Conjecture~\ref{con:general}, we conjecture that our finiteness result (Theorem~\ref{thm}) should also hold Hilbert modular forms of weights greater than or equal to two for all totally real fields of any fixed degree and all levels and all narrow ray class characters. 

\section{Preliminaries}
In this section, we set up the notations and recall some necessary notions and results on real quadratic fields that will be used in later sections.

Let $F=\mathbb Q(\sqrt{d})$ be a real quadratic field with $d>1$ being a squarefree integer. Let $\mathcal O=\mathcal O_F$ be the ring of integers of $F$, $\mathcal O^\times$ the group of units, $\mathfrak d$ the different of $F$, and $D$ the discriminant of $F$. Therefore $D=d$ if $d\equiv 1\imod 4$ and $D=4d$ otherwise. 
Let $\mathfrak p$ denote a prime ideal of $\mathcal O$, and $F_\mathfrak{p}$ and $\mathcal O_\mathfrak{p}$ be the completions of $F$ and $\mathcal O$ at $\mathfrak p$. For any fractional ideal $\mathfrak c$, considered as a $\mathbb Z$-lattice, we denote $\mathfrak c^\vee$ its dual lattice under the trace form of $F/\mathbb Q$; $\mathfrak c^\vee$ is also a fractional ideal. In particular, $\mathcal O^\vee=\mathfrak d^{-1}$.

We fix one real embedding of $F$ and for $a\in F$, we denote $a'$ the conjugate of $a$, which gives the other real embedding. Let $F_\mathbb{R}=F\otimes_\mathbb{Q}\mathbb R$, so $a\mapsto (a,a')$ gives the embedding $F\subset F_\mathbb{R}$. An element $x$ in $F_\mathbb{R}$, hence in $F$, is called totally positive if its two components are both positive; denoted by $x\gg 0$. For $A\subset F_\mathbb{R}$, we denote the subset of totally positive elements by $A^+$. 
Two fractional ideals $\mathfrak a,\mathfrak b$ are in the same narrow class if $\mathfrak a=(a)\mathfrak b$ for some $a\gg 0$ in $F^\times$.
We denote the narrow class number of $F$ by $h^+$.

Let $\mathbb A$, $\mathbb A^\times$, $\mathbb A_f$ and $\mathbb A_f^\times$ be the ring of adeles, the group of ideles, the ring of finite adeles and the group of finite ideles, respectively. We recall various characters. A Hecke character $\psi$ is a continuous character on $\mathbb A^\times/F^\times$ and $\psi=\prod_v\psi_v$ decomposes uniquely into local characters. We shall denote the induced character on $\mathbb A^\times$ also by $\psi$. An narrow ideal class character $\psi$ is a Hecke character that is trivial on the subgroup
$F^\times F_\mathbb{R}^+\prod_{\mathfrak p}\mathcal O_\mathfrak{p}^\times$. Equivalently, in terms of ideals, this is a character on the narrow ideal class group such that $\psi(a\mathcal O)=1$ for all $a\gg 0$ in $F$. There exists a unique pair $(r,r')\in \{0,1\}^2$, such that
\[\psi(a\mathcal O)=\text{sgn}(a)^{r}\text{sgn}(a')^{r'},\text{ for all } a\in F^\times.\] 
Note that in general not all sign vectors are associated to a narrow ideal class character. Since the narrow class group is abelian, we have precisely $h^+$ narrow ideal class characters.

The Dedekind zeta function for $F$ is defined as
\[\zeta_F(s)=\sum_{\mathfrak m}N(\mathfrak m)^{-s}=\prod_{\mathfrak p}(1-N(\mathfrak p)^{-s})^{-1},\]
where $\mathfrak m$ is over all nonzero integral ideals and $\mathfrak p$ is over all prime ideals in $\mathcal O$. In general, for any narrow ideal class character $\psi$, we define the Hecke L-function 
\[L(s,\psi)=\sum_{\mathfrak m}\psi(\mathfrak m)N(\mathfrak m)^{-s}=\prod_{\mathfrak p}(1-\psi(\mathfrak p)N(\mathfrak p)^{-s})^{-1}.\]
In particular, $\zeta_F(s)=L(s,1)$, where we denote the trivial character by $1$.
The series and the product for $L(s,\psi)$ are absolutely convergent for $\text{Re}(s)>1$, can be continued to be a meromorphic function on $\mathbb C$, and satisfies a functional equation. More precisely, assuming that the sign vector for $\psi$ is $(r,r')$, we have the following functional equation
\begin{equation}\label{eq0}
L(s,\psi)=W(\psi)(\pi^{-2}D)^{\frac{1}{2}-s}\frac{\Gamma\left(\frac{1-s+r}{2}\right)\Gamma\left(\frac{1-s+r'}{2}\right)}{\Gamma\left(\frac{s+r}{2}\right)\Gamma\left(\frac{s+r'}{2}\right)}L(1-s,\overline{\psi}),
\end{equation}
where $|W(\psi)|=1$ (See Corollary 8.6, Chapter VII in \cite{neukirch1999algebraic} for details).

The values of $L(s,\psi)$ at $1-k$ with $k\geq 2$, when $r=r'\equiv k\imod 2$, are given by
\begin{equation}
L(1-k,\psi)=W(\psi)\frac{2}{\pi}\left(\frac{D}{4\pi^2}\right)^{k-\frac{1}{2}}\Gamma(k)^2L(k,\overline{\psi}).
\end{equation}
In particular, $L(1-k,\psi)\neq 0$. Moreover, since for any $\psi$,
\[\zeta(4k)/\zeta^2(k)\leq \zeta_F(2k)/\zeta_F(k)\leq |L(k,\psi)|\leq \zeta_F(k)\leq \zeta^2(k), \quad k\geq 2,\] we have the bounds
\begin{equation}\label{eq1.4}
\frac{2}{\pi}\left(\frac{D}{4\pi^2}\right)^{k-\frac{1}{2}}\Gamma(k)^2\frac{\zeta(4k)}{\zeta^2(k)}\leq |L(1-k,\psi)|\leq \frac{2}{\pi}\left(\frac{D}{4\pi^2}\right)^{k-\frac{1}{2}}\Gamma(k)^2\zeta^2(k).
\end{equation}

\section{Hilbert Modular Forms}

We recall the classical and adelic Hilbert modular forms of full levels. It is well-known that the Eisenstein space vanishes if the weight is non-parallel (see, for example, \cite[Corollary in Section 1.4]{garrett1990holomorphic}), so we shall only consider parallel weights, since otherwise no such identities exist. Materials in this section can be found in \cite{garrett1990holomorphic} and \cite{shimura1978special}, and we note that our notion of congruence subgroups are more restrictive.

A (Hilbert) congruence subgroup $\Gamma$ is a subgroup of $\text{GL}_2(F)$ such that there exists an open compact subgroup $K\subset\text{GL}_2(\mathbb A_f)$ with $\Gamma=\text{GL}_2(F)\cap \text{GL}_2^+(F_\mathbb R)K$, where $+$ means the determinant is totally positive. It is clear that $\Gamma$ and $K$ determines each other.
For a fractional ideal $\mathfrak c$ and an integral ideal $\mathfrak n$ in $F$, we set
\[\Gamma_0(\mathfrak c,\mathfrak n)=
\left\{\gamma=\begin{pmatrix}
a&b\\c&d
\end{pmatrix}\in
\begin{pmatrix}
\mathcal O&\mathfrak c^{-1}\\
\mathfrak n\mathfrak c&\mathcal O
\end{pmatrix}\colon  \text{det}(\gamma)\in \mathcal{O}^{\times+} \right\}.\] Here $\mathfrak n$ is called the level. It is easy to see that $\Gamma_0(\mathfrak c,\mathfrak n)$ is a congruence subgroup and we denote the corresponding compact open subgroup by $K_0(\mathfrak c,\mathfrak n)$.  Denote 
\[\gamma^\iota=\begin{pmatrix}
d &-b\\
-c& a
\end{pmatrix}, \quad\text{if \ } \gamma=\begin{pmatrix}
a &b\\
c& d
\end{pmatrix},\]
and it defines an involution on $\text{Mat}_2(\mathbb A)$, under which $\Gamma_0(\mathfrak c,\mathfrak n)$ and $K_0(\mathfrak c,\mathfrak n)$ are invariant. 

We shall be only interested in the full-level groups $\Gamma_0(\mathfrak c,\mathcal O)$. Denote $\Gamma=\Gamma_0(\mathfrak c,\mathcal O)$ for the moment.
Let $\mathbb H^2=\{z=(z_1,z_2)\colon \text{Im}(z_i)>0, i=1,2\}$, and for any element
\[g=(g_1,g_2)=\left(\begin{pmatrix}
a_1 &b_1\\c_1&d_1
\end{pmatrix},\begin{pmatrix}
a_2 &b_2\\c_2&d_2
\end{pmatrix}\right)\in\text{GL}_2^+(F_\mathbb R),\]
set \[j(g,z)=(c_1z_1+d_1)(c_2z_2+d_2), \quad gz=\left(\frac{a_1z_1+b_1}{c_1z_1+d_1},\frac{a_2z_2+b_2}{c_2z_2+d_2}\right).\]
Via the embedding $\Gamma\subset \text{GL}_2^+(F_\mathbb R)$ by $\gamma\mapsto (\gamma,\gamma')$, we have an action of $\Gamma$ on $\mathbb H^2$; here $\gamma'$ is obtained by taking conjugates of all entries of $\gamma$.
A Hilbert modular form for $\Gamma$ of parallel weight $k\in\mathbb Z$, is a holomorphic function $f$ on $\mathbb H^2$
such that $f|_k\gamma (z)=f(z)$ for any $\gamma\in \Gamma$ and $z\in\mathbb H^2$; here the slash-$k$ operator (denoted $|_k$)is defined as
\[f|_{k}\gamma(z)=(\text{det}(\gamma\gamma'))^{\frac{{k}}{2}}j(\gamma,z)^{-{k}}f(\gamma z),\quad\text{with } \gamma=\begin{pmatrix}
a&b\\c&d
\end{pmatrix}.\]
We denote the space of such forms by $M_k(\Gamma)$. Any $f\in M_k(\Gamma)$ admits a Fourier expansion of the form
\[f(z)=\sum_{\mu\in (\mathfrak c^{-1})^\vee}a(\mu) \text{exp}\left(2\pi i\text{Tr}(\mu z)\right)=\sum_{\mu\in (\mathfrak c^{-1})^\vee}a(\mu) q^{\mu},\]
where $\text{Tr}(\mu z)=\mu z_1+\mu'z_2$, $q=(q_1,q_2)=(e^{2\pi iz_1},e^{2\pi iz_2})$, $q^\mu=q_1^{\mu}q_2^{\mu'}$. The \emph{Koecher principle}
says that $a(\mu)\neq 0 \Rightarrow \mu=0 \text{ or } \mu\gg 0$. Moreover for any $\varepsilon\in \mathcal O^{\times+}$ and any $\mu\in (\mathfrak c^{-1})^\vee$, we have $a(\varepsilon\mu)=N(\varepsilon)^{\frac{{k}}{2}}a(\mu)$. Similar results hold for all congruence subgroups.
We call $f\in M_k(\Gamma)$ cuspidal if $a_\gamma(0)=0$ for any $\gamma\in \text{GL}_2^+(F)$ with $a_\gamma(\mu)$ the Fourier coefficient of $f|_k\gamma$ (which is a Hilbert modular form for the congruence subgroup $\gamma^{-1}\Gamma\gamma$). The space of cusp forms is denoted by $S_k(\Gamma)$. The Petersson inner product is defined by
\[\langle f,h\rangle_\Gamma =\frac{1}{\nu(\Gamma\backslash \mathbb H^2)}\int_{\Gamma \backslash \mathbb H^2}f(z)\overline{h(z)}(y_1y_2)^kd\nu(z),\quad d\nu(z)=\prod_{j=1,2}\frac{dx_jdy_j}{y_j^2}, z_j=x_j+iy_j, j=1,2.\] With this, the Eisenstein space $E_k(\Gamma)$ is defined as the orthogonal complement of $S_k(\Gamma)$ in $ M_k(\Gamma)$. As in the elliptic case, the Petersson inner product is well-defined if one of the two components is cuspidal.

In general, Hecke theory is not available for $M_k(\Gamma)$ unless $h^+=1$. In order to explain the Hecke theory, we need adelic Hilbert modular forms. Now we fix $\Gamma=\Gamma_0(\mathcal O,\mathcal O)$ and $K=K_0(\mathfrak d,\mathcal O)$. Note that $K$ is not the compact open subgroup for $\Gamma$ and the shift by $\mathfrak d$ is for the correct definition of the normalized Fourier coefficients (see below). Set $K^+_{\infty}=(\mathbb R^\times \text{SO}_2(\mathbb R))^2$ and denote also by $i$ the element $(i,i)$ by abuse of notation. 
An adelic Hilbert modular form of weight $k$ for $\Gamma$ is a function $f:\text{GL}_2(\mathbb{A})\rightarrow\mathbb C$ such that the following properties hold:
\begin{enumerate}
\item $f(\gamma g u)=f(g)$ for all $\gamma\in \text{GL}_2(F)$, $g\in \text{GL}_2(\mathbb{A})$, and $u\in K$.
\item $f(gu_\infty)=(\det u_\infty)^{\frac{k}{2}}j(u_\infty,i)^{-k}f(g)$ for all $u_\infty\in K^+_{\infty}$ and $g\in \text{GL}_2(\mathbb{A})$.
\item For any $x\in \text{GL}_2(\mathbb A_{f})$, we define a function
$f_x:\mathbb{H}^n\rightarrow \mathbb{C}$ by 
$$
f_x(z)=(\det g)^{-\frac{k}{2}}j(g,i)^{k}f(xg)
$$ 
for $g_\infty\in \text{GL}_2^+(\mathbb{R})^2$ such that $g_\infty( i)=z$. Then $f_x$ is a holomorphic function.

\item Let $U$ be the unipotent radical of $\text{Res}_{F/\mathbb Q}\text{GL}_2$. An adelic Hilbert modular form $f$ is called a cusp form if 
$$
\int_{U(\mathbb Q)\backslash U(\mathbb{A}_\mathbb{Q})} f(ug)du=0,
$$ for all $g\in \text{GL}_2(\mathbb A)$, where $du$ is a Haar measure on $U(\mathbb A_\mathbb{Q})$.
\end{enumerate}
We denote the space of holomorphic and cuspidal adelic Hilbert modular forms by $\mathcal M_k$ and $\mathcal S_k$ respectively. Let $\psi$ be a narrow ideal class character and we say that $f\in \mathcal M_k$ has central character $\psi$ if $f(ag)=\psi(a)f(g)$ for each $a\in\mathbb A^\times$. The subspace with central character $\psi$ is denoted by $\mathcal M_k(\psi)$ and $\mathcal S_k(\psi)=\mathcal S_k\cap \mathcal M_k(\psi)$. 

We state the relation between these two versions of Hilbert modular forms. Let \[\{\mathfrak c_\nu :=t_\nu\mathcal O\}_{\nu=1}^{h^+}\] be a complete representatives set of the narrow class group of $F$, with $t_\nu$ being finite ideles. We shall assume that $t_1\mathcal O$ represents the identity narrow class. Set $\Gamma_\nu=\Gamma_0(\mathfrak c_\nu\mathfrak d,\mathcal O)$.
The Petersson inner product is defined by
\[\langle f, h\rangle=\sum_{\nu}\langle f_\nu,h_\nu\rangle_{\Gamma_\nu},\] with which we define the Eisenstein subspaces $\mathcal E_k$ and $\mathcal E_k(\psi)$ to be the orthogonal complement of $\mathcal S_k$ in $\mathcal M_k$ and $\mathcal S_k(\psi)$ in $\mathcal M_k(\psi)$ respectively. 

The following theorem is essentially a special case of Shimura's result \cite{shimura1978special}, where he treated general levels and general narrow ray class characters but did not give the precise definition of the adelic Hilbert modular forms explicitly. The proof is standard, and the argument for elliptic modular forms (\cite{gelbart1975automorphic}) can be carried over without difficulty. See also Demb{\'e}l{\'e} and Cremona's notes \cite{dembelemodular}.

\begin{Thm}[{\cite[Shimura]{shimura1978special}}]\label{iso}
There exist isomorphisms of complex vector spaces
\[
\mathcal{M}_{\mathbf{k}}\simeq \bigoplus_{\nu=1}^{h^+} M_{\mathbf{k}}(\Gamma_\nu)
,\quad
\mathcal{S}_{\mathbf{k}}\simeq \bigoplus_{\nu=1}^{h^+} S_{\mathbf{k}}(\Gamma_\nu)\quad \text{ and }\quad
\mathcal{E}_{\mathbf{k}}\simeq \bigoplus_{\nu=1}^{h^+} E_{\mathbf{k}}(\Gamma_\nu).
\] Moreover,
\[\mathcal{M}_{\mathbf{k}}= \bigoplus_{\psi} \mathcal{M}_{\mathbf{k}}(\psi),\quad
\mathcal{S}_{\mathbf{k}}= \bigoplus_{\psi} \mathcal{S}_{\mathbf{k}}(\psi)\quad \text{ and }\quad
\mathcal{E}_{\mathbf{k}}= \bigoplus_{\psi} \mathcal{E}_{\mathbf{k}}(\psi),\]
where in all sums $\psi$ runs through all $h^+$ narrow ideal class characters and some components may vanish.
\end{Thm}

Under such isomorphisms, we may write an element $f\in \mathcal{M}_{\mathbf{k}}$ as $f=(f_\nu)$ with $f_\nu\in  M_{\mathbf{k}}(\Gamma_\nu)$. For each integral ideal $\mathfrak m$, assuming that $\mathfrak m=t_\nu^{-1}(\mu)$ with $\mu\in (t_\nu\mathcal O)^+$, we define
\[c(\mathfrak m,f)=N(t_\nu)^{-\frac{k}{2}}a_\nu(\mu),\]
where $a_\nu(\mu)$ is the $\mu$-th normalized Fourier coefficient of $f_\nu$. This is clearly well-defined and we call it the $\mathfrak m$-th Fourier coefficient of $f$. The normalized constant term $c_\nu(0,f)$, for each $\nu$, is defined to be
\[c_\nu(0,f)=N(t_\nu)^{-\frac{k}{2}}a_\nu(0).\]

It is the space $ \mathcal{M}_{\mathbf{k}}$ that carries the Hecke theory. More precisely, for each integral ideal $\mathfrak m$, we have a Hecke operator $T_\mathfrak{m}$ on $ \mathcal{M}_{\mathbf{k}}$. The Hecke algebra generated by $T_\mathfrak{m}$ is commutative and normal and is also generated by $T_\mathfrak{p}$ for prime ideals $\mathfrak p$. The subspaces $\mathcal S_k$, $\mathcal E_k$, $\mathcal S_k(\psi)$ and $\mathcal E_k(\psi)$ are invariant under the Hecke algebra. A Hecke eigenform $f\in \mathcal{M}_{\mathbf{k}}$ is an eigenfunction for all $T_\mathfrak{m}$ and we call it normalized if $c(\mathcal O,f)=1$. For a normalized Hecke eigenform, the eigenvalue of $T_\mathfrak{m}$ is $c(\mathfrak m,f)$ for any $\mathfrak m$. The Hecke multiplicativity properties are similar to those in the case of elliptic modular forms. For example, if $f\in \mathcal M_k(\psi)$ is a normalized Hecke eigenform, then
$c(\mathfrak{mn},f)=c(\mathfrak{m},f)c(\mathfrak{n},f)$ if $(\mathfrak m,\mathfrak n)=1$, and if $\mathfrak p$ is a prime ideal, then
\begin{equation}\label{eq2.2}
c(\mathfrak p^2,f)=c(\mathfrak p,f)^2-\psi(\mathfrak p)N(\mathfrak p)^{k-1}.
\end{equation}
The following bound towards the \emph{generalized Ramanujan conjecture}, best so far, was obtained by Kim and Sarnak \cite{kim2003refined}: if $f\in \mathcal S_k$ is a normalized Hecke eigenform and $\mathfrak p$ is a prime ideal, then 
\begin{equation}\label{eq2.3}
|c(\mathfrak p,f)|\leq 2N(\mathfrak p)^{\frac{k-1}{2}+\frac{7}{64}}.
\end{equation}
This will be needed for the asymptotic behavior of two sides of some equations in the weights, which will obstruct the eigenform identities eventually.

\section{Product of Two Eigenforms}

Assume $k\geq 2$ from now on and keep other notations in the previous sections. We first recall a theorem of Shimura \cite{shimura1978special} on Eisenstein series. The computation of the constant terms is due to Dasgupta, Darmon and Pollack \cite{dasgupta2011hilbert}.

\begin{Thm}[{\cite[Proposition 3.4]{shimura1978special},\cite[Proposition 2.1]{dasgupta2011hilbert}}]\label{shimura} Let $k\geq 2$ and $\phi$ and $\psi$ be two narrow ideal class characters and assume that $(\phi_v\psi_v)(-1)=(-1)^k$ for both of the two real places $v$. There exists an element $E_k(\phi,\psi)\in \mathcal M_k(\phi\psi)$ such that 
\[c(\mathfrak m,E_k(\phi,\psi))=\sum_{\mathfrak r\mid \mathfrak m}\phi(\mathfrak m\mathfrak r^{-1})\psi(\mathfrak r)N(\mathfrak r)^{k-1},\] for all nonzero integral ideals $\mathfrak m$, and $E_k(\phi,\psi)$ is a normalized eigenform for $T_\mathfrak{m}$. Moreover, for each $\nu$, 
\[c_\nu(0,E_k(\phi,\psi))=2^{-2}\phi^{-1}(t_\nu)L(\phi^{-1}\psi,1-k).\]
\end{Thm}

\begin{Cor}\label{wiles}
For any narrow ideal class character $\psi$, the following set
\[\left\{E_k(\psi_1,\psi_2)\colon \psi_1\psi_2=\psi\right\}\]
a basis of $\mathcal E_k$
consisting of Hecke eigenforms.
\end{Cor}
Let $h$ denote the class number of $F$ in the following proof and note that $h$ stands for a Hecke eigenform elsewhere.
\begin{proof}
When $k=2$, this is done by Wiles \cite[Proposition 1.5]{wiles1986p}. Assume that $k>2$. Since the number of cusps is precisely $h$, by \cite[Theorem in Section 1.8]{garrett1990holomorphic}, we see that $E_k(\Gamma_\nu)$ has dimension $h$ for each $\nu$, so $\text{dim}(\mathcal E_k)=hh^+$, by Theorem \ref{iso}.  

On the other hand, there are precisely $h$ narrow class characters $\psi$ with $\psi_\infty(-1)=(-1)^k$, since it is a lift of a fixed character on
$F^\times F_\mathbb{R}\prod_{\mathfrak p}\mathcal O_\mathfrak{p}^\times$ to $\mathbb A^\times$, where the index is $h$. For each such character $\psi$, by Theorem \ref{shimura}, we have $h^+$ Eisenstein series $E_k(\psi_1,\psi_2)$. Since they are distinct Hecke eigenforms, they are linearly independent. This implies that $\text{dim}(\mathcal E_k(\psi))\geq h^+$, so $\text{dim}(\mathcal E_k)\geq hh^+$. This forces that $\text{dim}(\mathcal E_k(\psi))=h^+$ and the corollary follows.
\end{proof}

We shall need the following elementary lemma on normalized Fourier coefficients of the product of two Hilbert modular forms. 

\begin{Lem}\label{product} For $j=1,2$, let $k_j\in\mathbb Z$ and $\psi_j$ be a narrow ideal class character.
If $f=(f_\nu)\in \mathcal M_{k_1}(\psi_1)$ and $h=(h_\nu)\in \mathcal M_{k_2}(\psi_2)$,  then $f\cdot h=(f_\nu\cdot  h_\nu)\in\mathcal M_{k_1+k_2}(\psi_1\psi_2)$. Moreover, 
\begin{enumerate}
\item For each $\nu$, $c_\nu(0,f\cdot h)=c_\nu(0,f)c_\nu(0,h)$.
\item $
c(\mathcal O,f\cdot h)=c(\mathcal O,f)c_1(0,h)+c(\mathcal O,h)c_1(0,f)$.
\item $c((2),f\cdot h)=c_1(0,h)c((2),f)+c(\mathcal O,f)c(\mathcal O,h)+c_1(0,f)c((2),h)$.
\item If $(2)$ is inert, then for the ideal $(4)$, 
\begin{align*}c((4),f\cdot h)&=c_1(0,h)c((4),f)+c(\mathcal O,f)c((3),h)+c((2),f)c((2),h)+c(\mathcal O,h)c((3),f)\\&+c_1(0,f)c((4),h)
+\left\{\begin{matrix}
2c(\mathfrak d,f)+2c(\mathfrak d,h)&\text{if } D=5\\ 
0&\text{if } D\neq 5\\ 
\end{matrix}\right..\end{align*}
\item For the ideal $(3)$,
\begin{align*}c((3),f\cdot h)&=c_1(0,h)c((3),f)+c(\mathcal O,f)c((2),h)+c(\mathcal O,h)c((2),f)\\&+c_1(0,f)c((3),h)
+\left\{\begin{matrix}
2c(\mathfrak d,f)c(\mathfrak d,h)&\text{if } D=5\\ 
0&\text{if } D\neq 5\\ 
\end{matrix}\right..\end{align*}
\item If $(2)=\mathfrak p^2$ (ramifies) or $(2)=\mathfrak p\mathfrak p'$ (splits), then
\[c(\mathfrak p,f\cdot h)=c_\nu(0,h)c(\mathfrak p,f)+c_\nu(0,f)c(\mathfrak p,h),\quad \mathfrak p\sim t_\nu^{-1}\mathcal O.\]
\end{enumerate}

\end{Lem}
\begin{proof} 
Since it is clear that $f_\nu\cdot h_\nu\in M_{k_1+ k_2}(\Gamma_\nu)$, under the isomorphism, the tuple $(f_\nu\cdot h_\nu)$ determines a Hilbert modular form in $\mathcal M_{k_1+ k_2}$.  On the other hand, the function $f\cdot h$ is determined by
\[(f\cdot h)(\alpha_\nu g_\infty)=f(\alpha_\nu g_\infty)h(\alpha_\nu g_\infty)=f_\nu|_{k_1}g_\infty \cdot h_\nu|_{{k}_2}g_\infty=(f_\nu\cdot h_\nu)|_{k_1+k_2}g_\infty,\]
from which it follows that $f\cdot h=(f_\nu\cdot h_\nu)\in\mathcal M_{{k}_1+k_2}$, hence in $\mathcal M_{k_1+ k_2}(\psi_1\psi_2)$.

For ease of notations, we assume $t_1=1$, so $\mathfrak c_1=\mathcal O$. The formula for the constant Fourier coefficients follows directly from the definition and that for the $\mathcal O$-th terms follows from the fact that $1$ is minimal in the set $\mathcal O^+$ (of totally positive integers) under the partial order $\gg$. Indeed, for the component $f_1\cdot h_1$, the congruence subgroup is $\Gamma_0(\mathfrak d,\mathcal O)$ and $\mathfrak d^\vee=\mathcal O$ is the lattice where the Fourier expansion sums. Moreover, if $1=\mu_1+\mu_2$ with $\mu_1,\mu_2\in \mathcal O^+$, then 
\[1=(\mu_1+\mu_2)(\mu_1'+\mu_2')>\mu_1\mu_1'+\mu_2\mu_2'\geq 1+1= 2;\]
a contradiction and the formula follows.
For the ideal $(2)$, $\nu=1$. Then the Fourier expansion sums over $\mathcal O$ and we show that if $2=\mu_1+\mu_2$ inside $\mathcal O^+$, then we must have $\mu_1=\mu_2=1$. We see that \[4=N(2)=N(\mu_1+\mu_2)\geq N(\mu_1)+N(\mu_2)+2\sqrt{N(\mu_1)N(\mu_2)}\geq 1+1+2=4,\]
which forces $N(\mu_1)=N(\mu_2)=1$ and $\mu_1\mu_2'=\mu_1'\mu_2$. It follows that $\mu_1=\mu_2=1$.

We now consider the ideal $(4)$ when $(2)$ is inert. We first note that $D=d\equiv 1\imod 4$. Assume $4=\mu_1+\mu_2$ with $\mu_1,\mu_2\gg 0$ and
\[\mu_j=a_j+b_j\frac{1+\sqrt D}{2},\quad a_j,b_j\in\mathbb Z, j=1,2.\]
Since $4=\mu_1+\mu_2$, we have $b_1=-b_2$ and $a_1+a_2=4$. Moreover, since $\mu_j\gg 0$, we have
\[a_j+\frac{b_j}{2}>\frac{|b_j|}{2}\sqrt{D}, \quad j=1,2.\]
If $b_1=b_2=0$, then we have the three possibilities $(1,3)$, $(2,2)$ and $(3,1)$ for the pair $(\mu_1,\mu_2)$. Now we may assume that $b_1=-b_2>0$, and
the case when $b_1<0$ follows by switching $\mu_1$ and $\mu_2$. If $D\neq 5$, then $D\geq 13$. It follows that $2a_1>\sqrt {13}-1> 2$ and $2a_2>\sqrt {13}+1>4$, so $a_1\geq 2$ and $a_2\geq 3$. But $a_1+a_2=4$ and we have a contradiction. So if $D\neq 5$, we only have the above three possibilities.
If $D=5$, we first note that $b_1=1$, since otherwise $a_1>\sqrt{5}-1>1$ and $a_2>\sqrt{5}+1>3$. This implies that $a_1>0$ and $a_2>1$. Therefore, we have only two cases $(a_1,a_2)=(1,3)$ or $(2,2)$. So in total we have four more pairs for $(\mu_1,\mu_2)$:
\[\left(\frac{5+\sqrt{5}}{2},\frac{3-\sqrt{5}}{2}\right),\quad \left(\frac{5-\sqrt{5}}{2},\frac{3+\sqrt{5}}{2}\right),\quad \left(\frac{3+\sqrt{5}}{2},\frac{5-\sqrt{5}}{2}\right),\quad \left(\frac{3-\sqrt{5}}{2},\frac{5+\sqrt{5}}{2}\right).\]
This completes the case by noting that \[\left(\frac{5+\sqrt{5}}{2}\right)=\left(\frac{5-\sqrt{5}}{2}\right)=\mathfrak d, \quad\frac{3+\sqrt{5}}{2}, \frac{3-\sqrt{5}}{2}\in \mathcal O^\times.\]
The ideal $(3)$ can be taken care of similarly.

Now assume that $(2)=\mathfrak p^2$ ramifies or $(2)=\mathfrak p\mathfrak p'$ splits, and $\mathfrak p=t_\nu^{-1}(\mu)$ with $\mu\in (t_\nu\mathcal O)^+$. Then the Fourier expansion sums over $t_\nu\mathcal O$ and we show that $\mu$ is minimal among the totally positive elements in $t_\nu\mathcal O$. Indeed, assume otherwise and $\mu=\mu_1+\mu_2$ with $\mu_1,\mu_2$ totally positive. Note first that $N(\mu)=N(t_\nu)N(\mathfrak p)=2N(t_\nu)$. But
\[2N(t_\nu)=N(\mu_1+\mu_2)\geq N(\mu_1)+N(\mu_2)+2\sqrt{N(\mu_1\mu_2)}\geq 4N(t_\nu),\] which is impossible.
So $\mu$ is minimal and the formula for $c(\mathfrak p,f\cdot h)$ follows.
\end{proof}

We now prove Theorem~\ref{thm}. We separate the assertion of Theorem~\ref{thm} in two separate assertions. We assume that $f$ and $h$ are normalized Hecke eigenforms with the set of normalized Fourier coefficients \[\{c(\mathfrak m,f), c_\nu(0,f)\}\quad\text{ and }\quad\{c(\mathfrak m,h), c_\nu(0,h)\}\] respectively. Note that we are in full level case and all Hecke eigenforms are normalizable.
Clearly, we can divide it into two cases: $c_1(0,f)c_1(0,h)\neq 0$ or $c_1(0,f)c_1(0,h)=0$. Therefore, we have to prove the following two theorems, whose proof will be given in the next two sections.

\begin{Thm}\label{thm1} Among the solutions to the equation $g=f\cdot h$ in the Theorem~\ref{thm}, there are finitely many solutions with $c_1(0,f)c_1(0,h)\neq 0$.
\end{Thm} 

\begin{Thm}\label{thm2} Among the solutions to the equation $g=f\cdot h$ in the Theorem~\ref{thm}, there are finitely many solutions with $c_1(0,f)c_1(0,h)=0$.
\end{Thm} 

\section{Proof of Theorem \ref{thm1}} Assume that $f$ and $h$ are normalized Hecke eigenforms with $c_1(0,f)c_1(0,h)\neq 0$ and $g=f\cdot h$ is also a Hecke eigenform.  By Theorem \ref{shimura} and Corollary \ref{wiles}, they must be Eisenstein series and we may assume that
\[f=E_{k_1}(\phi_1,\psi_1)\quad \text{and}\quad h=E_{k_2}(\phi_2,\psi_2)\]
with $\phi_j$ and $\psi_j$ being narrow ideal class characters, $j=1,2$. Therefore, by Theorem \ref{shimura}, we have
\[c_1(0,f)=2^{-2}L(1-k_1,\phi_1^{-1}\psi_1) \quad\text{and}\quad c_1(0,h)=2^{-2}L(1-k_2,\phi_2^{-1}\psi_2).\] By Lemma \ref{product} we have $c_\nu(0,g)=c_\nu(0,f)c_\nu(0,h)$ and
\[c(\mathcal O,g)=c(\mathcal O,f)c_1(0,h)+c(\mathcal O,h)c_1(0,f)=c_1(0,h)+c_1(0,f).\]

Since $g$ is a Hecke eigenform, up to a nonzero scalar, $g$ is equal to $E_{k_1+k_2}(\phi,\psi)$ for some $\phi$ and $\psi$. By comparing the $\mathcal O$-th terms, we have
\[g=(c_1(0,f)+c_1(0,h))E_{k_1+k_2}(\phi,\psi).\] Then from the $\nu$-th constant terms, we derive that
\[\frac{c_\nu(0,f)c_\nu(0,h)}{c_1(0,f)+c_1(0,h)}=c_\nu(0,E_{k_1+k_2}(\phi,\psi)).\]
It follows, by Theorem \ref{shimura}, that for each $\nu$,
\[\phi_1(t_\nu)\phi_2(t_\nu)\left(\frac{1}{L(1-k_1,\phi_1^{-1}\psi_1)}+\frac{1}{L(1-k_2,\phi_2^{-1}\psi_2)}\right)=\phi(t_\nu)\frac{1}{L(1-k_1-k_2,\phi^{-1}\psi)}.\]
By considering the case $\nu=1$, we see that \begin{equation}\label{eq1}\phi_1(t_\nu)\phi_2(t_\nu)=\phi(t_\nu),\text{ for each }\nu,\end{equation}
\begin{equation}\label{eq4.2}
\frac{1}{L(1-k_1,\phi_1^{-1}\psi_1)}+\frac{1}{L(1-k_2,\phi_2^{-1}\psi_2)}=\frac{1}{L(1-k_1-k_2,\phi^{-1}\psi)}.
\end{equation}
It follows from \eqref{eq1} that $\phi_1\phi_2=\phi$, so $\psi=\psi_1\psi_2$, since $\phi\psi=\phi_1\phi_2\psi_1\psi_2$.

We now treat the case when $k_1\neq k_2$. We may assume that $k_1>k_2$. First note that, if $k_1$ is large, then
\[\left|\frac{L(1-k_1,\phi_1^{-1}\psi_1)}{L(1-k_2,\phi_2^{-1}\psi_2)}\right|\geq \left(\frac{D}{4\pi^2}\right)^{k_1-k_2}\frac{\Gamma(k_1)^2}{\Gamma(k_2)^2}\frac{\zeta(4k_1)}{\zeta(k_1)^2\zeta(k_2)^2}>\left(\frac{1}{4\pi^2}\right)^{k_1-k_2}\frac{\Gamma(k_1)^2}{\Gamma(k_2)^2}\frac{\zeta(4k_1)}{\zeta(k_1)^2\zeta(k_2)^2},\]
which in turn is bigger than $1$; indeed, if $k_2\geq k_1/2$, then 
\[\left(\frac{1}{4\pi^2}\right)^{k_1-k_2}\frac{\Gamma(k_1)^2}{\Gamma(k_2)^2}\geq \left(\frac{k_2^2}{4\pi^2}\right)^{k_1-k_2}>2,\] while if $k_2< k_1/2$, then
\[\left(\frac{1}{4\pi^2}\right)^{k_1-k_2}\frac{\Gamma(k_1)^2}{\Gamma(k_2)^2}\geq \left(\frac{1}{4\pi^2}\right)^{k_1-k_2}\frac{(k_1-1)!^2}{(k_2-1)!(k_1-k_2)!}\geq\frac{(k_1-1)!}{(4\pi^2)^{k_1-k_2}}>2.\] 
From this and by \eqref{eq1.4} and \eqref{eq4.2}, if $k_1$ is large, we have for some constant $C>0$ independent of $k_1,k_2$ and $D$,
\begin{align*}1&=\left|(L(1-k_1,\phi_1^{-1}\psi_1)+L(1-k_2,\phi_2^{-1}\psi_2))\frac{L(1-k_1-k_2,\phi^{-1}\psi)}{L(1-k_1,\phi_1^{-1}\psi_1)L(1-k_2,\phi_2^{-1}\psi_2)}\right|\\
&\geq  C\left(\frac{D}{4\pi^2}\right)^{k_2}\frac{\Gamma(k_1+k_2)^2}{\Gamma(k_1)^2}\left|\left(\frac{Dk_2^2}{4\pi^2}\right)^{k_1-k_2}\frac{\zeta(4k_1)}{\zeta^2(k_1)\zeta^2(k_2)}-1\right|\geq  C\left(\frac{D}{4\pi^2}\right)^{k_2}\frac{\Gamma(k_1+k_2)^2}{\Gamma(k_1)^2},
\end{align*} while this last expression can be arbitrarily large if $k_1$ is large since $\Gamma(k_1+k_2)\geq k_1^{k_2}\Gamma(k_2)$. 
For for each fixed pair $(k_1,k_2)$, this is also large, thus exceeds $1$ if $D$ is large. This finishes the case when $k_1\neq k_2$.

For the rest of the proof of Theorem \ref{thm1}, we assume that $k_1=k_2$, so $k=2k_1$.
Let us consider more normalized Fourier coefficients to complete the proof. 

\subsection{The case when $(2)=\mathfrak p$ is inert} In particular, $\mathfrak p$ is trivial in the narrow ideal class, so all of the narrow ideal class characters are trivial at $\mathfrak p$. In this case, by Lemma \ref{product}, after simplification and setting $k_1=k_2$, we have
\[\frac{4}{L(1-k_1,\phi_1^{-1}\psi_1)L(1-k_1,\phi_2^{-1}\psi_2)}+\frac{1+4^{k_1-1}}{L(1-k_1,\phi_1^{-1}\psi_1)}+\frac{1+4^{k_1-1}}{L(1-k_1,\phi_2^{-1}\psi_2)}=\frac{1+4^{2k_1-1}}{L(1-2k_1,\phi^{-1}\psi)},\] which, together with \eqref{eq4.2}, implies that
\begin{equation}\label{3}
\frac{4^{2k_1-1}-4^{k_1-1}}{L(1-2k_1,\phi^{-1}\psi)}=\frac{4}{L(1-k_1,\phi_1^{-1}\psi_1)L(1-k_1,\phi_2^{-1}\psi_2)}.
\end{equation} However, by \eqref{eq1.4}, for a constant $C>0$ that are independent of $k_1$ and $D$, we have
\[\left|\frac{4L(1-2k_1,\phi^{-1}\psi)}{L(1-k_1,\phi_1^{-1}\psi_1)L(1-k_1,\phi_2^{-1}\psi_2)}\right|\geq C\sqrt{D}(k_1-1)4^{2k_1}\geq C(k_1-1)4^{2k_1},\]
by the following Stirling's bound on the binomial coefficients
\[\begin{pmatrix}
2n\\n
\end{pmatrix}\geq n^{-\frac{1}{2}}2^{2n-1}.\] Therefore, this, together with \eqref{3}, implies that $k_1$ is bounded. For each such $k_1$, above inequalities also implies that $D$ is bounded, which finishes the proof in this case.

\subsection{The case when $(2)=\mathfrak p^2$ or $(2)=\mathfrak p\mathfrak p'$} Assume $\mathfrak p=t_\nu^{-1}(\mu)$. Again by Lemma \ref{product}, we have
\[\phi_1(t_\nu)\frac{\phi_1(\mathfrak p)+\psi_1(\mathfrak p)2^{k_1-1}}{L(1-k_1,\phi_1^{-1}\psi_1)}+\phi_2(t_\nu)\frac{\phi_2(\mathfrak p)+\psi_2(\mathfrak p)2^{k_1-1}}{L(1-k_1,\phi_2^{-1}\psi_2)}=\phi(t_\nu)\frac{\phi(\mathfrak p)+\psi(\mathfrak p)2^{2k_1-1}}{L(1-2k_1,\phi^{-1}\psi)},\] which, together with \eqref{eq4.2}, implies that
\begin{equation}\label{eq4.4}
\frac{B}{L(1-2k_1,\phi^{-1}\psi)}=\frac{A}{L(1-k_1,\phi_1^{-1}\psi_1)}, \quad \text{with}
\end{equation}
\begin{equation}
B=\phi(t_\nu)\psi(\mathfrak p)2^{2k_1-1}-\phi_2(t_\nu)\psi_2(\mathfrak p)2^{k_1-1},\quad
A=\phi_1(t_\nu)\psi_1(\mathfrak p)2^{k_1-1}-\phi_2(t_\nu)\psi_2(\mathfrak p)2^{k_1-1},
\end{equation}
since $t_\nu\mathfrak p=(\mu)$ and $\phi(t_\nu)=\phi(\mathfrak p)$ and the same holds for any narrow ideal class character.

\begin{Lem}\label{lem}
There exists a constant $C>0$, such that $|A|\geq CD^{-\frac{1}{2}}$ for all $D, k_1,\phi_j,\psi_j$, $j=1,2$.
\end{Lem}
\begin{proof} 
If $h^+=1$ or $2$, then $C$ is an integer and $|C|\geq 1$. If $h^+>2$, we see that
\[|A|\geq 2^{k_1-1}\left|1-e^{\frac{2\pi i}{h^+}}\right|.\]
If $2<h^+\leq 6$, then clearly $|A|\geq 2$. If $h^+>6$, we have
\[|A|\geq 2\cdot \sin\left(\frac{2\pi}{h^+}\right)\geq \frac{2\pi}{h^+}.\]
Recall the well-known trivial bound of class number of real quadratic fields: there exists a constant $C>0$, such that
$h^+\leq C\sqrt{D}$ for all $D$. It follows that $|A|\geq 2\pi C^{-1}D^{-\frac{1}{2}}$. Replacing $2\pi C^{-1}$ with $C$, we finish the proof.
\end{proof}

We continue the proof. By \eqref{eq1.4}, we see that
\[\left|\frac{L(1-2k_1,\phi^{-1}\psi)}{L(1-k_1,\phi_1^{-1}\psi_1)}\right|\geq C'D^{k_2}\frac{\Gamma(2k_1)^2}{\Gamma(k_1)^2}\geq C'D^{k_2}\Gamma(2k_1),\]
with $C'>0$ being a constant that is independent of $k_1,k_2$ and $D$. By Lemma \ref{lem},
\[|B|\geq C'CD^{k_2-\frac{1}{2}}\Gamma(2k_1)\geq C'C\Gamma(2k_1).\]
But $|B|\leq 2^{2k_1}$, which forces that there are only finitely many $k_1$. 
Now for each fixed $k_1$, such inequalities also shows that there can be only finitely many $D$, proving this case, hence Theorem \ref{thm1}.

We remark that in the ramified case, $A$ is an integer and hence $|A|\geq 1$. In the split case, we may apply the identity $c(\mathfrak p,g)c(\mathfrak p',g)=c((2),g)$. By lengthy but elementary computation, we may see that if $k_1$ is large, we must have 
\[\phi_1(t_\nu)\psi_1(\mathfrak p)=-\phi_2(t_\nu)\psi_2(\mathfrak p),\]
from which we also derive $|A|\geq 1$ in this case. In other words, we may avoid Lemma \ref{lem} and the class number bound.

\section{Proof of Theorem \ref{thm2}}
As before, let $f,h$ be normalized Hecke eigenforms and assume $g=f\cdot h$ is also a Hecke eigenform. To prove Theorem \ref{thm2}, assume that $c_1(0,f)c_1(0,h)=0$. We first note that if $c_1(0,f)=c_1(0,h)=0$, then by Lemma \ref{product}, we see that $c(\mathcal O,g)=0$ and $g$ is not a Hecke eigenform. So one of the factors is an Eisenstein series, thus consider only the parallel weight case. So, we may assume that 
\[c_1(0,f)\neq 0\quad\text{and}\quad c_1(0,h)=0\]
for the rest of this paper. We observe that $h$ necessarily lie in $\mathcal S_{k_2}(\psi_2)$ and $f=E_{k_1}(\phi_1,\psi_1)$ for some narrow ideal class characters $\phi_1,\psi_1,\psi_2$ by Theorem \ref{shimura}.

Since $c(\mathcal O,g)=c_1(0,f)$, we see that $c_1(0,f)^{-1}g$ is a normalized Hecke eigenform. Now by Lemma \ref{product}, we see that 
\begin{equation}\label{eq5.1}\frac{c((2),g)}{c_1(0,f)}=\frac{1}{c_1(0,f)}+c((2),h).\end{equation}
By Proposition 2.2 in \cite{shimura1978special}, we know that $c_1(0,f)^{-1}c((2),g)$ and $c((2),h)$ are algebraic integers, so is $\frac{1}{c_1(0,f)}$. But since \eqref{eq1.4} holds for any $\psi$, it gives a uniform bound for $L(1-k,\psi)^\sigma$ for all $\sigma\in\text{Gal}(\overline{\mathbb Q}/\mathbb Q)$. It follows that for some constant $C>0$,
\[\left|\frac{1}{c_1(0,f)^\sigma}\right|\leq C\left(\frac{4\pi^2}{D}\right)^{k_1-\frac{1}{2}}\frac{1}{\Gamma(k_1)^2}\leq C\frac{(4\pi^2)^{k_1-\frac{1}{2}}}{\Gamma(k_1)^2}\rightarrow 0,\quad \text{as } k_1\rightarrow\infty.\]
In particular, $|\frac{1}{c_1(0,f)^\sigma}|<1$ for all $\sigma$ if $k_1$ is large, in which case $\frac{1}{c_1(0,f)}$ is not an algebraic integer. The same holds for large $D$ with $k_1$ being fixed. This proves that for $g=f\cdot h$ to be a Hecke eigenform, there are only finitely many possibilities for $D$ and $k_1$, so there are only finitely many possible $\phi_1,\psi_1$ and $f$.

To finish the proof of Theorem \ref{thm2}, it suffices to show that for fixed $f$, there are only finitely many $h$ such that $g=f\cdot h$ is an eigenform. So, with $f$ fixed, we only have to show that $k_2$ is bounded. We will prove this in the following subsections. 

\subsection{The case when $(2)=\mathfrak p^2$ ramifies}
Suppose $\mathfrak p\sim t_\nu^{-1}\mathcal O$. Then by Lemma \ref{product}, we have
\[\frac{c(\mathfrak p,g)}{c_1(0,f)}=\phi_1(t_\nu)c(\mathfrak p,h).\]
This, together with \eqref{eq2.2} and \eqref{eq5.1}, implies that
\[\psi_2(\mathfrak p)2^{k_2-1}(1-(\phi_1\psi_1)(\mathfrak p)2^{k_1})=\frac{1}{c_1(0,f)}.\]
Clearly, this is impossible if $k_2$ is large.

\subsection{The case when $(2)=\mathfrak p\mathfrak p'$ splits} Suppose $\mathfrak p\sim t_\nu^{-1}\mathcal O$. Since $c(\mathfrak p,h)c(\mathfrak p',h)=c((2),h)$ and the same identity holds for $c_1(0,f)^{-1}g$, we have, by Lemma \ref{product} and \eqref{eq5.1},
\[\phi_1(t_\nu)c(\mathfrak p,h)\phi_1^{-1}(t_\nu)c(\mathfrak p',h)=\frac{1}{c_1(0,f)}+c(\mathfrak p,h)c(\mathfrak p',h),\] and $\frac{1}{c_1(0,f)}=0$, which is impossible.

\subsection{The final case when $(2)$ is inert} In this case, we need the $(4)$-th Fourier coefficients.
We first assume that $D\neq 5$. By Lemma \ref{product}, we have
\[\frac{c((2),g)}{c_1(0,f)}=c((2),h)+\frac{1}{c_1(0,f)},\]
\[\frac{c((4),g)}{c_1(0,f)}=c((4),h)+A,\text{ with }A=\frac{c((3),h)+c((3),f)+c((2),h)c((2),f)}{c_1(0,f)}.\]
Moreover, \[\frac{c((4),g)}{c_1(0,f)}=\left(\frac{c((2),g)}{c_1(0,f)}\right)^2-4^{k_1+k_2-1}\] and $c((4),h)=c((2),h)^2-4^{k_2-1}$. It follows that 
\[4^{k_2-1}(1-4^{k_1})=-\frac{1}{c_1(0,f)^2}-\frac{2c((2),h)}{c_1(0,f)}+A,\]
which is impossible when $k_2$ is large, since the right-hand side is bounded by $9^{\frac{k_2}{2}}$ up to a constant.

Finally we treat the case $D=5$. By Lemma \ref{product}, we have $c_1(0,f)^{-1}c((4),g)=c((4),h)+B$, with \[B=\frac{c((3),h)+c((3),f)+c((2),f)c((2),h)+2c(\mathfrak d,h)+2c(\mathfrak d,f)}{c_1(0,f)}.\]
One sees that $B$ is bounded by $9^{\frac{k_2}{2}}$ up to a constant, since $N(\mathfrak d)=5$. By the same argument as above, we have
\begin{equation}\label{eq5.2} 4^{k_2-1}(1-4^{k_1})=-\frac{1}{c_1(0,f)^2}-\frac{2c((2),h)}{c_1(0,f)}+B,\end{equation}
which is again not possible if $k_2$ is large.

This completes the proof of Theorem \ref{thm2}, hence that of Theorem~\ref{thm}.

\section{Eigenform Product Identities for $\mathbb Q(\sqrt 5)$}

In this section, we consider the concrete case $D=5$ and find the complete list of eigenform product identities.

The class number is $1$ for the field $\mathbb Q(\sqrt 5)$, and $(2)$ and $(3)$ both are inert. Since the fundamental unit is $\epsilon_0=\frac{1+\sqrt 5}{2}$ which has norm $-1$, we have $h^+=1$ and $\psi=1$. We shall drop the characters and denote $E_k=E_k(1,1)$.  The inequality \eqref{eq1.4} implies
\begin{equation}\label{eq6.1}
\frac{72}{\pi^5}\left(\frac{5}{4\pi^2}\right)^{k-\frac{1}{2}}\Gamma(k)^2\leq |\zeta_F(1-k)|\leq \frac{\pi^3}{18}\left(\frac{5}{4\pi^2}\right)^{k-\frac{1}{2}}\Gamma(k)^2,
\end{equation}
since $1<\zeta(k)\leq \zeta(2)=\frac{\pi^2}{6}$.

We look into the structure of $\mathcal M_k$ when $k$ is small. We need a theorem of Gundlach \cite{gundlach1963bestimmung} and we follow the notations in \cite[Theorem 1.39, 1.40]{bruinier2008hilbert}. Note that they considered the group $\text{SL}_2(\mathcal O)$ instead of $\Gamma=\Gamma_0(\mathfrak d,\mathcal O)$. In particular, $g_k=E_k|\alpha_0$ according to our notations and $s_k$ is a specific cusp form of weight $k$ for $\text{SL}_2(\mathcal O)$, where 
\[\alpha_0=\begin{pmatrix}
1&0\\
0 &\frac{5+\sqrt 5}{2}
\end{pmatrix}.\] 

\begin{Prop} \label{dim} (1) $\mathcal M_k=\{0\}$ if $k$ is odd and $\mathcal M_k=M_k(\text{SL}_2(\mathcal O))|_k\alpha_0^{-1}$ if $k$ is even.\\
(2) If $k<20$ is even, then $\mathcal M_k=M_k^\text{sym}(\text{SL}_2(\mathcal O))|_k\alpha_0^{-1}$. In particular, $\bigoplus_{k<20}\mathcal M_k$ is generated by monomials in $E_2, E_6$ and $E_{10}$, and we have the following table:
\begin{equation*}
\begin{array}{|c||c|c|c|c|c|c|}
\hline
\text{weight } k& 2&4&6&8&10&12\\
\hline
\text{dim}(\mathcal S_k) & 0&0&1&1&2&3\\
\hline
\end{array}
\end{equation*}
\end{Prop}
\begin{proof} We note first that $\mathcal M_k=M_k(\Gamma)$ (not $M_k(\text{GL}_2^+(\mathcal O))$) and $\mathfrak d=\left(\frac{5+\sqrt{5}}{2}\right)$. Therefore, 
\[\Gamma=\alpha_0\text{GL}_2^+(\mathcal O)\alpha_0^{-1},\]
and it follows that $\mathcal M_k=M_k(\text{GL}_2^+(\mathcal O))|_k\alpha_0^{-1}$.

Because $N(\epsilon_0)=-1$, from the definition of $\mathcal M_k$ by applying $\epsilon_0I$, we see that $\mathcal M_k=\{0\}$ if $k$ is odd. The same result holds for any $\mathbb Q(\sqrt{d})$ with a unit of norm $-1$. Note that this is not the case for $\text{SL}_2(\mathcal O)$.

If $k$ is even, we only have to prove that $M_k(\text{SL}_2(\mathcal O))=M_k(\text{GL}^+_2(\mathcal O))$. One inclusion is trivial and we assume now $f\in M_k(\text{SL}_2(\mathcal O))$. For any $\gamma\in \text{GL}^+_2(\mathcal O)$, since $\text{det}(\gamma)\gg 0$, we must have $\text{det}(\gamma)=\epsilon^2$ for some unit $\epsilon$. It follows that
\[f|_k\gamma=f\left|_k\gamma(\epsilon^{-1} I)(\epsilon I)\right.=f|_k(\epsilon I)=f,\]
because $k$ is even. We are done with (1).

By Gundlach's theorem, in notations of \cite[Theorem 1.40]{bruinier2008hilbert},  the graded algebra $M_*(\text{SL}_2(\mathcal O))$ is generated by $g_2, s_5, s_6$ and $s_{15}$. From which we see that if $k$ is even and $k<20$, then $\mathcal M_k=M_k^\text{sym}(\text{SL}_2(\mathcal O))$. Actually since only $s_5$ is skew-symmetric among the four generators, the smallest even weight when we can have a nonzero skew-symmetric Hilbert modular form happens at $k=20$, that is $s_5s_{15}$. By the structure of $M_{2*}^\text{sym}(\text{SL}_2(\mathcal O))$ given in \cite[Theorem 1.39]{bruinier2008hilbert}, the rest of the proposition follows easily.
\end{proof}

\begin{Lem}\label{eigen}
Let $h_6$ and $h_8$ be the only cuspidal normalized Hecke eigenforms of weight $6$ and $8$ respectively, and $h_{10}$, $h_{10}'$ be the two of weight $10$. We have the following Fourier coefficients for these Hecke eigenforms: 
\begin{equation*}
\begin{array}{|c||c|c|c|c|}
\hline
\mathfrak m & (2) & (3) & \mathfrak d & (4)\\
\hline
c(\mathfrak m,h_6)& 20 &90&-90& -624 \\
\hline
c(\mathfrak m,h_8)& 140 &3330&150& 3216\\
\hline
c(\mathfrak m,h_{10})& 170+30\sqrt{809} & 22590-540\sqrt{809}& 570-60\sqrt{809} & 494856+10200\sqrt{809} \\
\hline
c(\mathfrak m,h_{10}')& 170-30\sqrt{809} & 22590+540\sqrt{809}& 570+60\sqrt{809} & 494856-10200\sqrt{809} \\
\hline
\end{array}
\end{equation*}
\end{Lem}
\begin{proof} We first note that 
\[\frac{5+\sqrt 5}{2}=\mu_1+\mu_2,\quad \mu_1,\mu_2\in\mathcal O^+\]
has only two solutions \[\left(\frac{3+\sqrt 5}{2},1\right),\quad \left(1,\frac{3+\sqrt 5}{2}\right).\]
These decompositions are needed for dealing with the ideal $\mathfrak d$. 

Since $E_2\cdot E_4$ and $E_6$ have constant terms $(4\cdot 30\cdot 4\cdot 60)^{-1}$ and $67\cdot (4\cdot 630)^{-1}$, we must have
\[h_6=\frac{1}{60}\left(5360E_2\cdot E_4-7E_6\right).\] The Fourier coefficients of $h_6$ can be computed easily from Lemma \ref{product}. By Proposition \ref{dim}, we have $\text{dim}(\mathcal S_8)=1$ and $h_8=120E_2\cdot h_6$. The corresponding data follows easily from this.

For the weight $10$, it is easy to see that
\[h=\frac{39624096E_2\cdot E_8-3971E_{10}}{30126852}\]
is a normalized cusp form. Clearly $h'=120E_2\cdot h_8$ is also a normalized cusp form. We have the following table:
\begin{equation*}
\begin{array}{|c||c|c|c|c|}
\hline
\mathfrak m & (2) & (3) & \mathfrak d & (4)\\
\hline
c(\mathfrak m,h)& \frac{18087260}{119551} & \frac{2740912470}{119551} &\frac{72616890}{119551}& \frac{58400150256}{119551}\\
\hline
c(\mathfrak m,h')& 260 &20970&390& 525456\\
\hline
\end{array}
\end{equation*}
From this and the equation \eqref{eq2.2}, we have
\[h_{10}=ah+(1-a)h',\quad h_{10}'=a'h+(1-a')h'\]
with $a=119551(3-\sqrt{809})/433200$ and $a'$ its conjugate in $\mathbb Q(\sqrt{809})$. The normalized Fourier coefficients follow easily from this.
\end{proof}

Now we are ready to provide and prove the complete list of eigenform product identities when $D=5$.

\begin{Thm}\label{D=5}
The following two identities form the complete list of eigenform product identities $g=f\cdot h$ when $D=5$ and the weights are $2$ or greater (only one of $g=f\cdot h$ and $g=h\cdot f$ is counted):
\[E_4=60E_2^2,\quad h_8=120E_2\cdot h_6.\]
\end{Thm}
\begin{proof} We shall make use of the effective bounds in the proofs of Theorem \ref{thm1} and \ref{thm2}.

We first consider products of Eisenstein series. If $k_1>k_2$, for 
$|\zeta_F(1-k_1)/\zeta_F(1-k_2)|>1$, we need $k_1\geq 8$, by \eqref{eq6.1}. Using \eqref{eq6.1} and \eqref{eq4.2}, we have 
\begin{align*}1&=\left|(\zeta_F(1-k_1)+\zeta_F(1-k_2))\frac{\zeta_F(1-k_1-k_2)}{\zeta_F(1-k_1)\zeta_F(1-k_2)}\right|\\
&\geq \left(\frac{6}{\pi^4}\right)^2\left(\frac{5}{4\pi^2}\right)^{k_2}\frac{\Gamma(k_1+k_2)^2}{\Gamma(k_1)^2}\left|\left(\frac{6}{\pi^4}\right)^2\left(\frac{5k_2^2}{4\pi^2}\right)^{k_1-k_2}-1\right|.
\end{align*}
Using computer, the right-hand side larger than $1$ when $k_1\geq 8$. So we need to verify the cases $(k_1,k_2)=(4,2), (6,2)$ and $(6,4)$. Note that 
\[\zeta_F(-1)=\frac{1}{30},\quad \zeta_F(-3)=\frac{1}{60},\quad \zeta_F(-5)=\frac{67}{630}, \quad \zeta_F(-7)=\frac{361}{120}, \quad \zeta_F(-9)=\frac{412751}{1650},\]
and clearly \eqref{eq4.2} does not hold in any of these cases. Now we assume $k_1=k_2$. Using \eqref{eq6.1} and \eqref{3}, we have 
\begin{align*}1&\geq 1-4^{-k_1}=4^{1-2k_1}(4^{2k_1-1}-4^{k_1-1})\geq \frac{\pi}{2}\left(\frac{6}{\pi^4}\right)^3\left(\frac{5}{4\pi^2}\right)^{\frac{1}{2}}4^{2-2k_1}\frac{\Gamma(2k_1)^2}{\Gamma(k_1)^2}.
\end{align*}
The right-hand side is smaller than or equal to $1$ only when $k_1=k_2=2$ or $4$. The former gives the identity $E_2^2=\frac{1}{60}E_4$ that holds trivially, while the latter case is impossible since \eqref{eq4.2} does not hold by above zeta values.

Now we consider the case when $h$ is cuspidal. Firstly, since $4\zeta_F(1-k_1)^{-1}$ is integral by \eqref{eq5.1}, $|\zeta_F(1-k_1)|\leq 4$. By \eqref{eq6.1}, we have
\[\frac{72}{\pi^5}\left(\frac{5}{4\pi^2}\right)^{k_1-\frac{1}{2}}\Gamma(k_1)^2\leq 4.\] Such inequality only happens when $k_1=2,4,6$ or $8$. Since $4\zeta_F(1-k_1)^{-1}$ is integral, from the actual zeta values above, we see that $k_1$ can only be $2$ or $4$.

We first assume that $k_1=2$. Then $f=E_2$ and by Theorem \ref{shimura}, 
\[c_1(0,f)=\frac{1}{4}\zeta_F(-1)=\frac{1}{120}, \quad c(\mathfrak d,f)=6,\quad c((2),f)=5, \quad c((3),f)=10. \] Then \eqref{eq5.2} implies
\[15\cdot 4^{k_2-1}\leq 120^2+120\cdot(2\cdot 3^{k_2-\frac{25}{32}}+5\cdot 2^{k_2-\frac{25}{32}}+4\cdot 5^{\frac{k_2-1}{2}+\frac{7}{64}}-22).\] This holds only if $k_2\leq 10$. From Proposition \ref{dim}, $k_2$ can only be $6,8$ or $10$. The identity $E_2\cdot h_6=\frac{1}{120}h_8$ holds trivially, while $E_2\cdot h_8$ is not an eigenform by the proof of Lemma \ref{eigen}. We need to consider $120E_2\cdot h_{10}$ and $120E_2\cdot h_{10}'$. Since we may obtain one from the other by taking conjugate in $\mathbb Q(\sqrt{809})$, we just need to consider $h=120E_2\cdot h_{10}$. By the table in Lemma \ref{eigen}, we easily see that \eqref{eq5.2} does not hold and $h$ is not an eigenform.

Finally, let $k_1=4$. We have $f=E_4$ and similarly by Theorem 3.1,
\[c_1(0,f)=\frac{1}{4}\zeta_F(-3)=\frac{1}{240}, \quad c(\mathfrak d,f)=126,\quad c((2),f)=65, \quad c((3),f)=730. \] Then \eqref{eq5.2} implies
\[255\cdot 4^{k_2-1}\leq 240^2+240\cdot (2\cdot 3^{k_2-\frac{25}{32}}+126\cdot 2^{k_2-\frac{25}{32}}+4\cdot 5^{\frac{k_2-1}{2}+\frac{7}{64}}-982),\] which holds only if $k_2\leq 8$. Therefore, $k_2$ can only be $6$ or $8$. Since $E_4\cdot h_6$ is a scalar multiple of $E_2\cdot h_8$, it is not an eigenform from the previous case. Again, for $E_4\cdot h_8$ from the table in Lemma \ref{eigen}, we check that \eqref{eq5.2} does not hold, forcing $E_4\cdot h_8$ not to be an eigenform. This completes the proof.
\end{proof}

\section{A General Conjecture}\label{ss:conj}
In the light of Theorem~\ref{thm} and of \cite{johnson2013hecke} the following conjecture is natural.
\begin{conj}\label{con:general}
Let $n\geq 1$ be an integer. Then amongst all totally real fields $F/\Q$ with $[F:\Q]=n$ and all nonzero integral ideals $\mathfrak n$, there exist only finitely many solutions to the equation
$$g=f\cdot h,$$
where $g,f,h$ are Hecke eigenforms of level $\mathfrak n$ and integral weights $2$ or greater.	
\end{conj}	

It also natural to ask if the hypothesis on the degree of the totally real fields considered in conjecture~\ref{con:general} is necessary. In other words, perhaps the total number of such identities amongst all totally real fields is finite. But this may be too optimistic at this juncture.

\bibliographystyle{amsplain}
\bibliography{paper}

\providecommand{\bysame}{\leavevmode\hbox to3em{\hrulefill}\thinspace}
\providecommand{\MR}{\relax\ifhmode\unskip\space\fi MR }
\providecommand{\MRhref}[2]{%
  \href{http://www.ams.org/mathscinet-getitem?mr=#1}{#2}
}
\providecommand{\href}[2]{#2}
\begin{thebibliography}{10}

\bibitem{beyerl2012factoring}
Jeff Beyerl, \emph{{On factoring Hecke eigenforms, nearly holomorphic modular
  forms, and applications to L-values}}, All Dissertations,Paper 891 (2012).

\bibitem{bruinier2008hilbert}
Jan~H. Bruinier, \emph{Hilbert modular forms and their applications}, The 1-2-3
  of modular forms, Universitext, Springer, Berlin, 2008, pp.~105--179.

\bibitem{dasgupta2011hilbert}
Samit Dasgupta, Henri Darmon, and Robert Pollack, \emph{{Hilbert modular forms
  and the Gross-Stark conjecture}}, Annals of mathematics \textbf{174} (2011),
  no.~1, 439--484.

\bibitem{dembelemodular}
Lassina Demb{\'e}l{\'e} and John Cremona, \emph{Modular forms over number
  fields}, Expository notes.

\bibitem{duke1999product}
William Duke, \emph{{When is the product of two Hecke eigenforms an
  eigenform}}, Number theory in progress \textbf{2} (1999), 737--741.

\bibitem{emmons2005products}
Brad~A. Emmons, \emph{{Products of Hecke eigenforms}}, Journal of Number Theory
  \textbf{115} (2005), no.~2, 381--393.

\bibitem{garrett1990holomorphic}
Paul~B. Garrett, \emph{Holomorphic hilbert modular forms}, Wadsworth \&
  Brooks/Cole Advanced Books \& Software, 1990.

\bibitem{gelbart1975automorphic}
Stephen~S. Gelbart, \emph{Automorphic forms on adele groups}, no.~83, Princeton
  University Press, 1975.

\bibitem{ghate2000monomial}
Eknath Ghate, \emph{On monomial relations between eisenstein series}, Journal
  of the Ramanujan Mathematical Society \textbf{15} (2000), no.~2, 71--80.

\bibitem{ghate2002products}
\bysame, \emph{On products of eigenforms}, Acta Arithmetica \textbf{102}
  (2002), 27--44.

\bibitem{gundlach1963bestimmung}
Karl-Bernhard Gundlach, \emph{{Die bestimmung der funktionen zur Hilbertschen
  modulgruppe des zahlk{\"o}rpers $\mathbb Q (\sqrt 5)$}}, Mathematische
  Annalen \textbf{152} (1963), no.~3, 226--256.

\bibitem{johnson2013hecke}
Matthew~L. Johnson, \emph{Hecke eigenforms as products of eigenforms}, Journal
  of Number Theory \textbf{133} (2013), no.~7, 2339--2362.

\bibitem{kim2003refined}
Henry~H. Kim and Peter Sarnak, \emph{{Refined estimates towards the Ramanujan
  and Selberg conjectures}}, J. Amer. Math. Soc \textbf{16} (2003), no.~1,
  175--181.

\bibitem{neukirch1999algebraic}
J\"{u}rgen Neukirch and Norbert Schappacher, \emph{Algebraic number theory},
  vol.~9, Springer Berlin, 1999.

\bibitem{shimura1978special}
Goro Shimura, \emph{{The special valuesof the zeta functions associated with
  Hilbert modular forms}}, Duke Mathematical Journal (1978), 637--679.

\bibitem{wiles1986p}
Andrew Wiles, \emph{On p-adic representations for totally real fields}, Annals
  of Mathematics (1986), 407--456.

\end{thebibliography}

\end{document}